\documentclass[12pt]{article}

\usepackage{amsmath}
\usepackage{amssymb}
\usepackage{amsthm}
\usepackage{bbm}
\usepackage{graphicx}
\usepackage{color}
\usepackage{epsfig}

\newtheorem{theorem}{Theorem}

\newtheorem{prop}[theorem]{Proposition}
\newtheorem{lemma}[theorem]{Lemma}
\newtheorem{cor}[theorem]{Corollary}

\newtheorem{question}{Question}

\newcommand{\F}{{\mathbb F}}
\newcommand{\N}{{\mathbb N}}
\newcommand{\Z}{{\mathbb Z}}

\newcommand{\cH}{{\mathcal H}}

\newcommand{\cQ}{{\mathcal Q}}

\newcommand{\cS}{{\mathcal S}}
\newcommand{\cT}{{\mathcal T}}

\newcommand{\schur}{\textsc{schur}}
\newcommand{\pyth}{\textsc{pyth}}
\newcommand{\prim}{\textsc{prim}}

\allowdisplaybreaks

\title{Pythagorean Partition-Regularity and Ordered Triple Systems with the Sum Property}
\author{Joshua Cooper\footnote{University of South Carolina, Department of Mathematics, \texttt{cooper@math.sc.edu}} \;and Chris Poirel\footnote{Virginia Tech, Department of Computer Science, \texttt{chris.poirel@gmail.com} }}
\date{\today}

\begin{document}

\maketitle

\begin{abstract} Is it possible to color the naturals with finitely many colors so that no Pythagorean triple is monochromatic?  This question is even open for two colors.  A natural strategy is to show that some small nonbipartite triple systems cannot be realized as a family of Pythagorean triples. It suffices to consider partial triple systems (PTS's), and it is therefore natural to consider the Fano plane, the smallest nonbipartite PTS.  We show that the Pythagorean triples do not contain any Fano plane.  In fact, our main result is that a much larger family of ``ordered'' triple systems (viz. those with a certain ``sum property'') do not contain any Steiner triple system (STS).
\end{abstract}

An equation over the integers is called ``partition regular'' if, for any coloring of the naturals (or integers) with finitely many colors, some solution to the equation is monochromatic.  For example, it is the first nontrivial case of Van der Waerden's famous Theorem that $2y = x+z$ is partition regular; this is another way to state that any coloring of $\N$ contains arithmetic progressions of length three in one color.  Other celebrated results are Schur's Theorem that $x+y=z$ is partition regular, and its broad generalization, Rado's Theorem.  Many more interesting examples and their offshoots are discussed in \cite{LR04}.

Much less is known about the regularity of nonlinear equations.  A very natural question is the regularity of the ``Pythagorean equation'' $x^2 + y^2 = z^2$.  Is there a way to color $\N$ with finitely many colors so that no Pythagorean triple is monochromatic?  It is not hard to see that there exists an $O(\log n)$-coloring of $\{1,\ldots,n\}$ (color $k$ by the number of times $5$ divides $k$), but essentially nothing more is known about this question.  It is not even known if it is possible to $2$-color the naturals with no monochromatic Pythagorean triple!  In fact, by using a combination of random and greedy heuristics -- and hundreds of hours of processor time -- the authors have found a way to $2$-color the integers $1$ to $1344$ so that no Pythagorean triple with all three points in that range are monochromatic.  This coloring is shown in Figure \ref{fig:coloring}.

\begin{figure}[h]
\centering
\includegraphics[scale = .4]{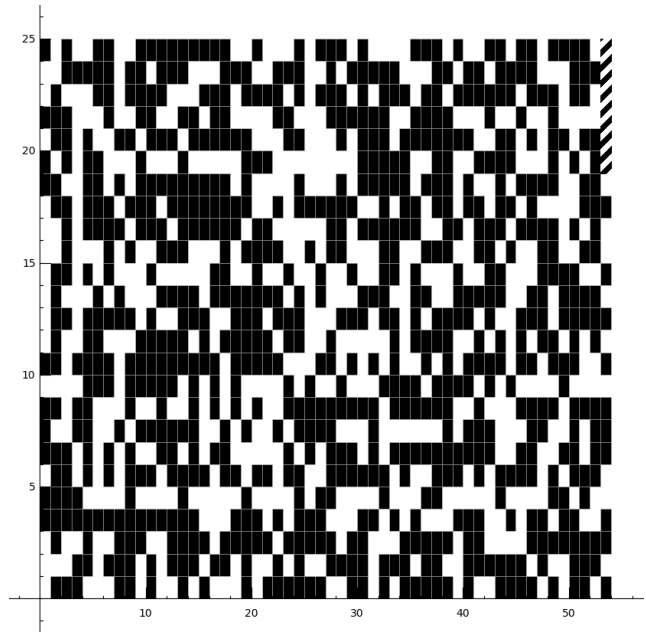}
\caption{Graphical representation of a $2$-coloring of $\{1,\ldots,1344\}$ avoiding monochromatic Pythagorean triples.  The integer $n$ is represented by a block at location $(a,b)$ where $n = 25b + a$, $1 \leq a \leq 25$.  The striped region is uncolored.} \label{fig:coloring}
\end{figure}

In order to show that the Pythagorean triples are partition regular, it suffices to demonstrate a subfamily of them which is not bipartite, i.e., any $2$-coloring induces a monochromatic triple.  (The existence of such a family follows from the de Bruijn-Erd\H{o}s Theorem, q.v. e.g. \cite{O83}.)  It is ``well known'' that no two Pythagorean triples share two points, i.e., no leg and hypotenuse of one right triangle are the two legs of another integer right triangle.  Indeed, the two equations $x^2 + y^2 = z^2$ and $y^2 + z^2 = t^2$ yield two squares ($y^2$ and $z^2$) whose sum and difference are also squares.  One can find an elementary proof that no such pair of squares exists in Sierpi\'nski's classic text (\cite{S88}).

Therefore, in our search for nonbipartite subsystems of triples, we need only consider those in which each pair of triples intersect in at most one point.  Such families are known as ``linear $3$-uniform hypergraphs'' (in the graph theory community) and ``partial triple systems'' or ``packings'' (in design theory).  The smallest nonbipartite partial triple system is the ubiquitous Fano plane $F_7$, the projective plane over $\F_2$ with seven points and seven triples.  $F_7$ is also a ``Steiner triple system,'' meaning that each pair of points is contained in exactly one triple.  Our main result below is that a broad class of triple systems including the family $\pyth$ of all Pythagorean triples contains no Steiner triple systems.  Therefore, any search for nonbipartite partial triple systems in $\pyth$ necessarily must consider other systems larger than $F_7$.

In the next section, we introduce some notation and definitions.  Section \ref{ref:noSTS} contains the proof that $\pyth$, and all ordered triple systems with the ``sum property'', contain no Steiner triple systems.  We conclude by mentioning a few outstanding questions about $\pyth$ we would like to see answered.

\section{Preliminaries}

A {\it triple system} $\cH$ is a pair $(V,E) = (V(\cH),E(\cH))$ consisting of a vertex set $V$ and a family of unordered triples $E \subset \binom{V}{3}$.    A {\it Steiner triple system} is a partial triple system $\cS$ so that, for each pair of vertices $x$ and $y$, there is some $z$ so that $\{x,y,z\}$ is an edge of $\cS$.  An {\it ordered triple system} $\cH$ is a triple $(V,E,<) = (V(\cH),E(\cH),<_{\cH})$ consisting of a vertex set $V$, a family of unordered triples $E \subset \binom{V}{3}$, and a total ordering $<$ on $V$.  We say that an ordered triple system $\cH$ has the {\it sum property} if, whenever $\{a,b,c\}$ and $\{a^\prime,b^\prime,c^\prime\}$ are two edges with respective maxima $c$ and $c^\prime$,
$$
(a \leq a^\prime) \wedge (b < b^\prime) \Rightarrow (c < c^\prime).
$$

\noindent {\bf Examples}:
\begin{enumerate}
\item Let $V = \Z$, and, for each $a \neq b$, let $e = (a,b,a+b)$ be an edge of $E$.  We call this the {\it Schur Triple System}, and denote it by $\schur$.  It is easy to see that $\schur$ has the sum property.  This shows immediately that the sum property does not imply finite colorability, since Schur's Theorem says that any coloring $\chi : \N \rightarrow [c]$ with finitely many colors $c$ admits a monochromatic $e \in \schur$, i.e., $|\chi(e)| = 1$.  Such a map is called a (weak) hypergraph coloring.   A map $\chi : \N \rightarrow [c]$ so that $|\chi(e)| = 3$ instead of just $|\chi(e)| > 1$ is known as a ``strong coloring,'' and corresponds exactly to a proper vertex coloring of the complement of the ``leave,'' i.e., the graph of all pairs contained in some triple $e \in E(\cH)$.  Graph and set theorists sometimes call this the ``shadow'' graph of $\cH$, and  denote it by the topological boundary operator $\partial \cH$.  Topologists variously know this graph as the ``$1$-skeleton'' as well.
\item It is clear that any order-preserving isomorphic image of a triple system with the sum property also has the sum property, and that any subhypergraph of a triple system with the sum property does as well.  For example, we define the {\it Pythagorean Triple System} $\pyth$ by $V = \N$ and $E(\pyth) = \{\{a,b,c\} : a^2 + b^2 = c^2\}$.  Since we can embed $\pyth$ into $\schur$ monotonically by the map $n \mapsto n^2$, $\pyth$ has the sum property as well.  It is a wide open problem to determine whether $\pyth$ has a weak coloring.  (It is even open whether it is strongly colorable.)  As mentioned in the introduction, $\pyth$ is actually linear, i.e., no two edges intersect in more than one vertex.
\item A special subsystem of $\pyth$ is the {\it Primitive Pythagorean Triple System} $\prim$ consisting of all Pythagorean triples which are relatively prime.  That is, $V = \N$ and
    $$
    E(\prim) = \{\{a,b,c\} : a^2 + b^2 = c^2 \textrm{ and } \gcd(a,b,c) = 1\}.
    $$
    It is easy to see that $\pyth$ is actually a union of dilates of $\prim$ by each $d \in \N$.  However, $\prim$ is bipartite: the parity coloring $n \mapsto n \pmod{2}$ provides a $2$-coloring.
\end{enumerate}

We define a special class of partial triple systems called ``bicycles.''  The $k$-bicycle has $2k+2$ vertices and $2k$ edges.  Its vertices are the elements of $\Z_{2n}$ and two ``antipodes'' $a$ and $b$; its edges are all triples of the form $\{a,2j,2j+1\}$ and $\{b,2j-1,2j\}$, $0 \leq j < n$.  The $2$-bicycle is also known as the {\it Pasch configuration}, or {\it quadrilateral}: the six-point partial triple system consisting of the edges $abc$, $ade$, $bef$, $cdf$.  The $3$-bicycle appears in the literature as the ``hexagon'' (e.g., \cite{CF08}): eight points $\{a,b,d,e,f,g,h,i\}$ with edges $\{afh,aei,adg,beh,bdi,bfg\}$.

\begin{figure}[h]
\centering
\includegraphics[scale = .4]{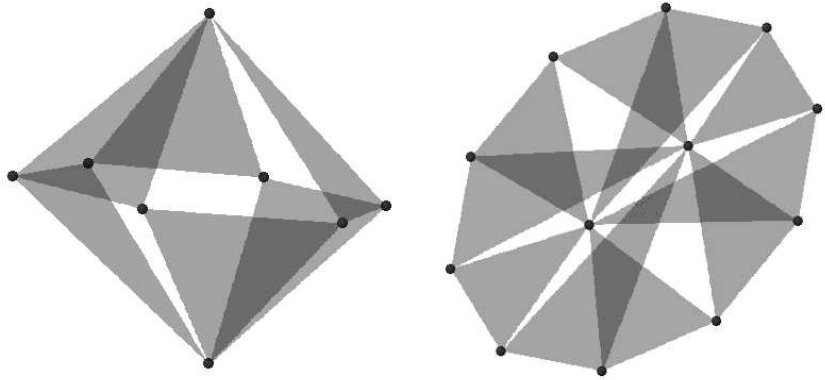}
\caption{A $3$-bicycle and a $5$-bicycle.} \label{fig:bicycles}
\end{figure}

The following proposition follows immediately from well-known results in the theory of triple systems.  (See, for example, \cite{CCR83}.)  For completeness, we give a short proof.

\begin{lemma} \label{prop:stshasbicycle} If $v,w$ are vertices of nontrivial Steiner triple system $\cS$, then there is a $k$-bicycle for some $k \geq 2$ in $\cS$ with antipodes $v$ and $w$.
\end{lemma}
\begin{proof} Define the ``link'' $\cS_x$ of a vertex $x \in \cS$ (also known as the ``star'' of $x$) to be the set of pairs $\{a,b\}$ so that $\{a,b,x\}$ is an edge of $\cS$.  For some $z \in \cS$, $\{v,w,z\}$ is a triple of $\cS$, so $\cS_v$ consists of a matching $M_1$ on $\cS - \{v,w\}$, plus the edge $\{w,z\}$; similarly, $\cS_w$ consists of a matching $M_2$ on $\cS - \{v,w\}$, plus the edge $\{v,z\}$.  The union of $M_1$ and $M_2$ is composed of even-length cycles of length at least $4$; any one of these forms a bicycle with antipodes $v$ and $w$.
\end{proof}

We now define two weaker implicants of the sum property.  We say that an ordered triple system has the {\it upper sum property} if, whenever $\{a,b,c\}$ and $\{a,b^\prime,c^\prime\}$ are two edges with $c$ and $c^\prime$ their respective maxima,
$$
(b > b^\prime) \Rightarrow (c > c^\prime).
$$
This clearly follows from the sum property by setting $a = a^\prime$.  We say that an ordered triple system {\it lower sum property} if, whenever $\{a,b,c\}$ and $\{a^\prime,b^\prime,c\}$ are two edges with $c$ as both of their maxima,
$$
(a \geq a^\prime) \Rightarrow (b < b^\prime).
$$
To see that the lower sum property follows from the sum property, suppose that $a \geq a^\prime$ but $b > b^\prime$.  Then it follows that the maximal element of $\{a,b,c\}$ is greater than the maximal element of $\{a^\prime,b^\prime,c\}$, whence $c > c$, a contradiction.  We say that two pairs of integers $\{a,b\}$ and $\{c,d\}$ are ``nesting'' if $a \leq c \leq d \leq b$.  Then it is possible to restate the upper sum property as the fact that, for each vertex $x$, the subset of the link graph $\cH_x$ intersecting $\{y : y \geq x\}$ is a non-nesting matching.  The lower sum property may be similarly restated as the fact that, for each vertex $x$, the subset of the link graph $\cH_x$ induced by $\{y : y \leq x\}$ is fully nested matching, i.e., for each two edges $e$ and $f$, $e$ is nested in $f$ or vice versa.

\section{Sum Property implies No STS} \label{ref:noSTS}

\begin{prop} \label{prop:noquadupper} If $\cH$ has the upper sum property, and $\cQ$ is a $k$-bicycle in $\cH$, then the maximal two points of $\cQ$ are not its antipodes.
\end{prop}
\begin{proof} Let
$$
\cQ = (\Z_{2n} \cup \{a,b\},\{a01,a23,a45,\ldots\}\cup\{b12,b34,b56,\ldots\}) \subset \cH
$$
be a $k$-bicycle, and suppose $a$ and $b$ are the maximal two points of $\cQ$.  We may assume without loss of generality that $a > b$.  Then the maximal elements of all triples are $a$ or $b$ (depending on which antipode they contain).  Therefore, since $\{a,2j,2j+1\} \cap \{b,2j+1,2j+2\} = \{2j+1\}$,
$$
(a > b) \Rightarrow (2j > 2j+2)
$$
for each $0 \leq j < n$.  However, the quantities above are modulo $2n$, whence the set of resulting inequalities is circular and therefore inconsistent.
\end{proof}

\begin{prop} \label{prop:noquadlower} If $\cH$ has the lower sum property, and $\cQ$ is a $k$-bicycle in $\cH$, then the maximal two points of $\cQ$ are not its antipodes.
\end{prop}
\begin{proof} Let
$$
\cQ = (\Z_{2n} \cup \{a,b\},\{a01,a23,a45,\ldots\}\cup\{b12,b34,b56,\ldots\}) \subset \cH
$$
be a $k$-bicycle, and suppose $a$ and $b$ are the maximal two points of $\cQ$.  We may assume without loss of generality that $a > b$.  Then the maximal elements of all triples are $a$ and $b$ (depending on which antipode they contain).  Since $\{a,2j,2j+1\}$ is an edge for each $0 \leq j < n$, the $n$ pairs $\{2j,2j+1\}$ are a matching, and they are linear ordered by nesting, by the lower sum property.  Suppose, without loss of generality, that $\{0,1\}$ is the outermost matching edge and $0 < 1$, so that, for all $x \in \Z_{2n} \setminus \{0,1\}$, $0 < x < 1$.  The pairs $\{-1,0\}$ and $\{1,2\}$ are also nested, since they arise from the edges $\{b,-1,0\}$ and $\{b,1,2\}$.  Hence, $-1 < \{1,2\} < 0$ or $1 < \{-1,0\} < 2$, and each of these possibilities contradicts $0 < 1$.
\end{proof}

\begin{cor} If $\cH$ has the full, lower, or upper sum property, then it does not contain any Steiner triple system.
\end{cor}
\begin{proof} Suppose $\cH$ contained some Steiner triple system $\cT$.  Let $a$ and $b$ be the maximal elements of $\cT$.  Then $a$ and $b$ are the antipodes of some bicycle, by Lemma \ref{prop:stshasbicycle}.  However, this contradicts Propositions \ref{prop:noquadupper} and/or \ref{prop:noquadlower}.
\end{proof}

Note that a triple system with the sum property can actually contain a quadrilateral: for example, $\schur$ contains $\{5,15,20\}$, $\{5,8,13\}$, $\{7,8,15\}$, $\{7,13,20\}$.

\section{Further Questions}

Here we collect a few open questions regarding $\pyth$ that we consider interesting.

\begin{question} Since the sum properties imply the nonexistence of some configurations with chromatic number $3$, is it possible to exploit them to find a coloring with few colors?
\end{question}

\begin{question} Is $\pyth$ strongly colorable?  Or does it contain arbitrarily large cliques?
\end{question}

\begin{question} Does $\pyth$ contain $F_7$ minus a single triple?
\end{question}

\begin{question} Is $\pyth$ connected?  Is $\prim$ connected?
\end{question}

\begin{question} Are there any non-finitely colorable {\it linear} triple systems with the sum property?
\end{question}

\section{Acknowledgements}

Thanks to the University of South Carolina Magellan Project for funding this research, and to Leonard Bailey for the proof that $\pyth$ is linear.


\begin{thebibliography}{100}
\bibitem{CCR83} Charles J.~Colbourn, Marlene J.~Colbourn, Alexander Rosa, Completing small partial triple systems.  {\it Discrete Math.}  {\bf 45}  (1983), no. 2-3, 165--179.
\bibitem{CF08} Charles J.~Colbourn and Yuichiro Fujiwara, Small stopping sets in Steiner triple systems, {\it Crypt. Commun.}, to appear.
\bibitem{LR04} Bruce M.~Landman and Aaron Robertson, Ramsey theory on the integers. Student Mathematical Library, {\bf 24}. American Mathematical Society, Providence, RI, 2004.
\bibitem{O83} Oystein Ore, Theory of graphs. American Mathematical Society Colloquium Publications, Vol. XXXVIII American Mathematical Society, Providence, R.I. 1983.
\bibitem{S88} Wac\l aw Sierpi\'nski, Elementary theory of numbers. North-Holland Mathematical Library, 31. North-Holland Publishing Co., Amsterdam; PWN---Polish Scientific Publishers, Warsaw, 1988.
\end{thebibliography}
\end{document}